\documentclass[11pt,reqno]{amsart}
\usepackage{amssymb}
\usepackage{amsmath}
\usepackage{amsthm}
\usepackage{graphicx}
\usepackage{caption}
\usepackage{bm}
\usepackage[colorlinks=true,urlcolor=blue,
citecolor=red,linkcolor=blue,linktocpage,pdfpagelabels,
bookmarksnumbered,bookmarksopen]{hyperref}%per colore-link
\usepackage[titletoc,title]{appendix}
\numberwithin{equation}{section} %riparte da zero ogni sezione
\newcounter{cont}[section] 

\newtheorem{thm}[cont]{Theorem}
\newtheorem{prop}[cont]{Proposition}
\newtheorem{lem}[cont]{Lemma}

\theoremstyle{definition}

\theoremstyle{remark}
\newtheorem{rem}[cont]{Remark}

\begin{document}
\baselineskip=16pt

\title[Slow motion for nonlinear damped hyperbolic Allen--Cahn systems]{Slow motion for one-dimensional \\ nonlinear damped hyperbolic Allen--Cahn systems}

\author[R. Folino]{Raffaele Folino}
\address[Raffaele Folino]{Dipartimento di Ingegneria e Scienze dell'Informazione e Matematica, Universit\`a degli Studi dell'Aquila (Italy)}
\email{raffaele.folino@univaq.it}


\maketitle

\begin{abstract} 
We consider a nonlinear damped hyperbolic reaction-diffusion system in a bounded interval of the real line with homogeneous Neumann boundary conditions
and we study the metastable dynamics of the solutions. 
Using an ``energy approach'' introduced by Bronsard and Kohn \cite{Bron-Kohn} to study slow motion for Allen--Cahn equation
and improved by Grant \cite{Grant} in the study of Cahn--Morral systems, 
we improve and extend to the case of systems the results valid for the hyperbolic Allen--Cahn equation (see \cite{Folino}). 
In particular, we study the limiting behavior of the solutions as $\varepsilon\to0^+$, where $\varepsilon^2$ is the diffusion coefficient, and
we prove existence and persistence of metastable states for a time $T_\varepsilon>\exp(A/\varepsilon)$.
Such metastable states have a \emph{transition layer structure} and the transition layers move with exponentially small velocity.
\end{abstract}

\section{Introduction}
The goal of this paper is to study the metastable dynamics of the solutions to the \emph{nonlinear damped hyperbolic Allen--Cahn system}
\begin{equation}\label{hyp-al-ca-sys}
	\tau \bm{u}_{tt}+G(\bm{u})\bm{u}_t=\varepsilon^2\bm{u}_{xx}+\bm{f}(\bm{u}), \qquad \quad x\in [a,b], \; t>0,
\end{equation}
where $\bm{u}(x,t)\in\mathbb R^m$ is a vector-valued function, $G:\mathbb R^m\rightarrow\mathbb R^{m\times m}$ is a matrix valued function of several variables,
$\bm{f}:\mathbb R^m\rightarrow\mathbb R^m$ is a vector field and $\varepsilon,\tau$ are positive parameters.
Precisely, we are interested in the limiting behavior of the solutions as $\varepsilon\to0^+$,
and we study existence and persistence of metastable states for \eqref{hyp-al-ca-sys}.

System \eqref{hyp-al-ca-sys} is complemented with homogeneous Neumann boundary conditions
\begin{equation}\label{Neumann-sys}
	\bm{u}_x(a,t)=\bm{u}_x(b,t)=0, \qquad\quad \forall t>0,
\end{equation}
and initial data 
\begin{equation}\label{cond-iniz-sys}
	\bm{u}(x,0)=\bm{u}_0(x), \qquad \bm{u}_t(x,0)=\bm{u}_1(x), \qquad \qquad x\in[a,b].
\end{equation} 
We assume that $\bm{f},G$ are smooth functions with $G$ a positive-definite matrix for all $\bm{u}\in\mathbb R^m$,
that is there exists a constant $\alpha>0$ such that 
\begin{equation}\label{g-pos-def}
	G(\bm{u})\bm{v}\cdot\bm{v}\geq\alpha|\bm{v}|^2, \qquad \quad \forall\, \bm{u},\bm{v}\in\mathbb R^m.
\end{equation}
Regarding $\bm{f}$, we suppose that it is a gradient field and $\bm{f}(\bm{u})=-\nabla F(\bm{u})$ 
where $F\in C^3(\mathbb R^m,\mathbb R)$ is a nonnegative function with a finite number ($K\geq2$) of zeros, namely
\begin{equation}\label{F1}
	F(\bm{u})\geq0 \quad \forall \, \bm{u}\in\mathbb R^m, \qquad  \textrm{ and } \quad F(\bm{u})=0 \Longleftrightarrow \bm{u}\in\{\bm{z}_1,\dots,\bm{z}_K\}.
\end{equation}
Moreover, we assume that the Hessian $\nabla^2F$ is positive definite at each zero of $F$:
\begin{equation}\label{F2}
	\nabla^2F(\bm{z}_j)\bm{v}\cdot\bm{v}>0 \quad \textrm{ for all } \, j=1,\dots,K \, \textrm{ and }  \, \bm{v}\in\mathbb R^m\backslash\{0\}. 
\end{equation}
Therefore, $\bm{z}_1,\dots,\bm{z}_K$ are global minimum points of $F$ and stable stationary points for system \eqref{hyp-al-ca-sys}.

In the scalar case $m=1$, system \eqref{hyp-al-ca-sys} becomes
\begin{equation}\label{scalar}
	\tau u_{tt}+g(u)u_t=\varepsilon^2u_{xx}+f(u),
\end{equation}
with $g$ a strictly positive smooth function and $f=-F'$, where the potential $F$ is a nonnegative function with $K$ zeros at $z_1,\dots,z_K$:
$F(z_j)=F'(z_j)=0$ and $F''(z_j)>0$ for any $j=1,\dots,K$.
In the case $K=2$, $F$ is a double-well potential with non-degenerate minima of same depth, and $f$ is a bistable reaction term.
The simplest example is $F(u)=\frac14(u^2-1)^2$, which has two minima in $-1$ and $+1$.

Equation \eqref{scalar} is a hyperbolic variation of the classic \emph{Allen--Cahn equation}
\begin{equation}\label{allencahn}
	u_t=\varepsilon^2u_{xx}+f(u),
\end{equation}
that is a reaction-diffusion equation (of parabolic type), proposed in \cite{Allen-Cahn} to describe the motion of antiphase boundaries in iron alloys.
Reaction-diffusion equations (of parabolic type) undergo the same criticisms of the linear diffusion equation, 
mainly concerning lack of inertia and infinite speed of propagation of disturbances.
To avoid these unphysical properties, many authors proposed hyperbolic variations of the classic reaction-diffusion equation, 
that enter in the framework of \eqref{scalar} for different choices of $g$;
for instance, for $g(u)\equiv1$, we have a \emph{damped nonlinear wave equation}, 
that is the simplest hyperbolic modification of \eqref{allencahn}.
A different hyperbolic modification is obtained by substituting the classic Fick's diffusion law (or Fourier law) with a relaxation relation of Cattaneo--Maxwell type 
(see \cite{Cat}, or \cite{JP89a,JP89b}); in this case, the damping coefficient is $g(u)=1-\tau f'(u)$ 
and if $f=-F'$ with $F$ a double-well potential with non-degenerate minima of same depth,
we have the \emph{Allen--Cahn equation with relaxation} (see \cite{Folino, FLM}). 
Equation \eqref{scalar} has also a probabilistic interpretation: in the case without reaction ($f=0$), it describes a \emph{correlated random walk} 
(see Goldstein \cite{Gol}, Kac \cite{Kac}, Taylor \cite{Tay} and Zauderer \cite{Zau}).

A complete list of papers devoted to equation \eqref{scalar} would be prohibitive; far from being exhaustive,
here we recall some works where the derivation of equation \eqref{scalar} was studied in different contexts: 
Dunbar and Othmer \cite{DunbOthm86},  Hadeler \cite{Hade99},
Holmes \cite{Holmes}, and Mendez et al. \cite{MenFedHor}.  
We also recall that existence and stability of traveling fronts for equation \eqref{scalar}
in the case of bistable reaction term is provided in \cite{Gallay-Joly} for $g\equiv1$,
and in \cite{LMPS} for the Allen--Cahn equation with relaxation, i.e $g=1-\tau f'$.

In analogy to the relaxation case of \eqref{scalar}, let us consider the particular case of \eqref{hyp-al-ca-sys} corresponding to the choice $G(\bm{u})=\mathbb{I}_m-\tau \bm{f}'(\bm{u})$,
where $\bm{f}'(\bm{u})$ is the Jacobian of $\bm{f}$ evaluated at $\bm{u}$.
We call it the {\it ``one-field'' equation} of system
\begin{equation}\label{nonlinearCat}
	\begin{cases}
		\bm{u}_t+\bm{v}_x=\bm{f}(\bm{u}),\\
		\tau \bm{v}_t+\varepsilon^2\bm{u}_x=-\bm{v},
	\end{cases}
\end{equation}
obtained after eliminating the $\bm v$ variable. 
Note that, for $\tau=0$, we formally obtain the reaction-diffusion system
\begin{equation}\label{readiff}
	\bm{u}_t=\varepsilon^2\bm{u}_{xx}+\bm{f}(\bm{u}).
\end{equation}
Some properties (long time behavior, invariance principles, Turing instabilities) 
of systems of the form \eqref{nonlinearCat} with general reaction term $\bm{f}$
have been studied by Hillen in \cite{Hillen1,Hillen2,Hillen3}.

The aim of this paper is twofold:
firstly, we will extend to the case of systems the slow motion results valid for the hyperbolic Allen--Cahn equation \eqref{scalar} (see \cite{Folino});
secondly, we will improve the energy approach used in \cite{Folino} to obtain an exponentially large lifetime of the metastable states.   

Metastable dynamics is characterized by evolution so slow that (non-stationary) solutions \emph{appear} to be stable;
metastability is a broad term describing the persistence of unsteady structures for a very long time.
For the Allen--Cahn model \eqref{allencahn}, this phenomenon was firstly observed in \cite{Bron-Kohn,Carr-Pego,Carr-Pego2,Fusco-Hale}.
In particular, Bronsard and Kohn \cite{Bron-Kohn} introduced an ``energy approach'', 
based on the underlying variational structure of the equation, to study the metastable dynamics of the solutions.
We also recall the study of generation, persistence and annihilation of metastable patterns performed in \cite{Chen}.
In this work, the author studied the persistence of the metastable states by using a different approach, 
known as ``dynamical approach'', proposed by Carr--Pego \cite{Carr-Pego} and Fusco--Hale~\cite{Fusco-Hale}.
In \cite{Bel-Na-No}, the authors provide a variational counterpart of the dynamical results of \cite{Carr-Pego,Fusco-Hale}.
They justify and confirm, from a variational point of view, the results of \cite{Carr-Pego,Fusco-Hale} 
on the exponentially slow motion of the metastable states.

The dynamical approach and the energy one can be adapted and extended to the hyperbolic variation \eqref{scalar}. 
In \cite{FLM}, by using the dynamical approach, the authors show the existence of an ``approximately invariant'' 
$N$-dimensional manifold $\mathcal{M}_{{}_0}$ for the hyperbolic Allen--Cahn equation:
if the initial datum is in a tubular neighborhood of $\mathcal{M}_{{}_0}$, the solution remains in such neighborhood for an exponentially long time.
Moreover, for an exponentially long time, the solution is a function with $N$ transitions between $-1$ and $+1$ (the minima of $F$) 
and the transition points move with exponentially small velocity.
On the other hand, in \cite{Folino}, by using the energy approach, it is proved that if the initial datum $u_0$ has a \emph{transition layer structure}
and the $L^2$--norm of the initial velocity $u_1$ is bounded by $C\varepsilon^{\frac{k+1}2}$,
then in a time scale of order $\varepsilon^{-k}$ nothing happens, and the solution maintains the same number of transitions of its initial datum.

The phenomenon of metastability is present in a very large class of different evolution PDEs. 
It is impossible to quote all the contributes, here we recall that using a similar approach to \cite{Carr-Pego,Fusco-Hale}, 
slow motion results have been proved for the Cahn--Hilliard equation in \cite{AlikBateFusc91,Bates-Xun1,Bates-Xun2}.
The energy approach is performed in \cite{Bron-Hilh} for the classical Cahn--Hilliard equation and in \cite{FLM19} for its hyperbolic variation.
We also recall the study of metastability for scalar conservation laws \cite{FLMS17,KreiKrei86,LafoOMal95,Mascia-Strani,ReynWard95,Strani3},
convection-reaction-diffusion equation \cite{Strani2},
general gradient systems \cite{Otto-Rez}, high-order systems \cite{Kal-VdV-Wan}.

The aforementioned bibliography is confined to one-dimensional scalar models; 
the papers \cite{Be-Or-Sm,Be-Sm,Strani} deal with the extension to the case of systems of the results valid for the scalar reaction-diffusion equations.
In particular, in \cite{Be-Or-Sm} a system of reaction-diffusion equations is considered in the whole real line,
with the reaction term $\bm{f}=-\nabla F$ and $F$ satisfying \eqref{F1}-\eqref{F2};
in \cite{Be-Sm} is considered the degenerate case, that is when $F$ satisfies \eqref{F1}, but not the condition \eqref{F2}. 
Strani \cite{Strani} studied systems of the form \eqref{readiff} in a bounded interval, where $\bm f=-\nabla F$ and $F$ satisfying \eqref{F1}-\eqref{F2} with two distinct minima.
On the other hand, Grant \cite{Grant} extended to Cahn--Morral systems the slow motion results of the Cahn--Hilliard equation, 
by improving the energy approach of Bronsard and Kohn \cite{Bron-Kohn}.
The improvement from superpolynomial to exponential speed is made possible by incorporating some ideas of Alikakos and McKinney \cite{AlikMcKi}
and some techniques of Sternberg \cite{Sternberg}. 
In this paper we use these ideas to improve and extend to the system \eqref{hyp-al-ca-sys} the results of \cite{Folino}.
The key point to apply the energy approach of Bronsard and Kohn in the system \eqref{hyp-al-ca-sys} is the presence
of the modified energy functional 
\begin{equation}\label{energy/eps}
	E_\varepsilon[\bm{u},\bm{u}_t](t):=\frac{\tau}{2\varepsilon}\|\bm u_t(\cdot,t)\|^2_{{}_{L^2}}+P_\varepsilon[\bm{u}](t),
\end{equation}
where 
\begin{equation*}
	\|\bm u_t(\cdot,t)\|^2_{{}_{L^2}}:=\int_a^b |\bm{u}_t(x,t)|^2dx \quad \mbox{ and } \quad 
	P_\varepsilon[\bm{u}](t):=\int_a^b\left[\frac\varepsilon2 |\bm{u}_x(x,t)|^2+\frac{F(\bm{u}(x,t))}\varepsilon\right]dx.
\end{equation*}
The modified energy functional defined in \eqref{energy/eps} is a nonincreasing function of time $t$ along the solutions of \eqref{hyp-al-ca-sys}-\eqref{Neumann-sys}.
Indeed, if $\bm{u}$ is a solution of \eqref{hyp-al-ca-sys} with homogeneous Neumann boundary conditions \eqref{Neumann-sys}, then
\begin{equation}\label{energy/eps-var}
	\varepsilon^{-1}\int_0^T\int_a^b G(\bm{u})\bm{u}_t\cdot\bm{u}_t\,dxdt=E_\varepsilon[\bm{u},\bm{u}_t](0)-E_\varepsilon[\bm{u},\bm{u}_t](T).
\end{equation}
The proof of \eqref{energy/eps-var} is in Appendix \ref{well-posed} (see Proposition \ref{prop-var-ener}). 
It follows that the assumption on $G$ implies the dissipative character of the system \eqref{hyp-al-ca-sys}. 
In particular, using \eqref{g-pos-def} and \eqref{energy/eps-var}, we obtain
\begin{equation}\label{dissipative}
	\varepsilon^{-1}\alpha\int_0^T\int_a^b |\bm{u}_t|^2 dxdt\leq E_\varepsilon[\bm{u},\bm{u}_t](0)-E_\varepsilon[\bm{u},\bm{u}_t](T).
\end{equation}
Note that the functional $P_\varepsilon$ is the modified energy functional for the parabolic case \eqref{readiff}
and we have a new term concerning the $L^2$--norm of $\bm u_t$ in the hyperbolic case.
Inequality \eqref{dissipative} is crucial in the use of the energy approach.
Let us briefly explain the strategy of this approach.
We remark that every piecewise constant function $\bm v$ assuming values in $\{\bm z_1,\dots,\bm z_K\}$ is a stationary solution of \eqref{hyp-al-ca-sys} with $\varepsilon=0$.
Fix $\bm{v}:[a,b]\rightarrow\{\bm{z}_1,\dots,\bm{z}_K\}$ having exactly $N$ jumps located at $a<\gamma_1<\gamma_2<\cdots<\gamma_N<b$. 
In this way, we fix the number of the \emph{transition layers} and their positions in $(a,b)$ as $\varepsilon\rightarrow0$.
Fix $r$ so small that $(\gamma_i-r,\gamma_i+r)\subset[a,b]$ for any $i$ and 
\begin{equation*}
	(\gamma_i-r,\gamma_i+r)\cap(\gamma_j-r,\gamma_j+r)=\emptyset, \qquad \quad \mbox{ for } i\neq j.
\end{equation*}
In the scalar case $m=1$, the minimum energy to have a transition between the two equilibrium points $-1$ and $+1$ is $c_0:=\displaystyle\int_{-1}^{+1}\sqrt{2F(s)}\,ds$. 
In general, for $m\geq1$, from Young's inequality and the positivity of the term $\frac{\tau}{2\varepsilon}\|\bm u_t\|^2_{{}_{L^2}}$, it follows that
\begin{equation}\label{E-Young}
	E_\varepsilon[\bm{u},\bm{u}_t](t)\geq P_\varepsilon[\bm{u}](t)\geq\sqrt2\int_a^b\sqrt{F(\bm{u}(x,t))}|\bm{u}_x(x,t)|dx.
\end{equation}
This justifies the use of the modified energy \eqref{energy/eps}; 
indeed, the right hand side of inequality \eqref{E-Young} is strictly positive and does not depend on $\varepsilon$. 
For \eqref{E-Young}, we assign to the discontinuous function $\bm{v}$ the asymptotic energy
\begin{equation*}
	P_0[\bm{v}]:=\sum_{i=1}^N\phi(\bm{v}(\gamma_i-r),\bm{v}(\gamma_i+r)),
\end{equation*}
where
\begin{equation*}
\phi(\bm{\xi}_1,\bm{\xi}_2):=\inf\left\{J[\bm{z}]: \bm{z}\in \emph{AC}([a,b],\mathbb R^m), \bm{z}(a)=\bm{\xi}_1, \bm{z}(b)=\bm{\xi}_2\right\}
\end{equation*}
and
\begin{equation*}
	J[\bm{z}]:=\sqrt2\int_a^b\sqrt{F(\bm{z}(s))}|\bm{z}'(s)|ds.
\end{equation*}	
It is easy to check that $\phi$ is a metric on $\mathbb R^m$. 
Moreover, Young's inequality and a change of variable imply that 
\begin{equation*}
	P_\varepsilon[\bm{z};c,d]\geq\phi(\bm{z}(c),\bm{z}(d)),
\end{equation*}
for all $a\leq c<d\leq b$, where we use the notation $P_\varepsilon[\bm{z};c,d]$, 
when the integral in \eqref{energy/eps} is over the interval $[c,d]$ instead of $[a,b]$. 
From \eqref{E-Young}, it follows that $P_0[\bm v]$ is the minimum energy to have $N$ transitions between the equilibrium points $\bm z_1,\dots,\bm z_K$.

When $\varepsilon>0$ the function $\bm v$ is not a stationary solution of \eqref{hyp-al-ca-sys};
we consider an initial datum $\bm{u}_0\in H^1([a,b])^m$ that is close to $\bm v$ in $L^1$ for $\varepsilon$ small. 
Precisely, we assume that the initial data $\bm{u}_0$, $\bm{u}_1$ depend on $\varepsilon$ and
\begin{equation}\label{L1}
	\lim_{\varepsilon\rightarrow 0} \|\bm{u}_0^\varepsilon-\bm{v}\|_{{}_{L^1}}=0.
\end{equation}
Moreover, we assume that for all $\varepsilon\in(0,\hat\varepsilon)$, at the time $t=0$, the modified energy \eqref{energy/eps} satisfies
\begin{equation}\label{energy0}
	E_\varepsilon[\bm{u}_0^\varepsilon, \bm{u}_1^\varepsilon]\leq P_0[\bm{v}]+C\exp(-A/\varepsilon),
\end{equation}
for some constants $A,C>0$. 
We say that a family of functions $\bm{u}_0^\varepsilon$ has a \emph{transition layer structure} if $\bm{u}_0^\varepsilon$ satisfies \eqref{L1} and 
\begin{equation*}
	P_\varepsilon[\bm{u}_0^\varepsilon]\leq P_0[\bm{v}]+C\exp(-A/\varepsilon).
\end{equation*}
Notice that $P_\varepsilon[\bm{u}_0^\varepsilon]\leq E_\varepsilon[\bm{u}_0^\varepsilon, \bm{u}_1^\varepsilon]$, 
thus \eqref{L1} and \eqref{energy0} imply that $\bm{u}_0^\varepsilon$ has a transition layer structure.
We will prove that there exists a time $T_\varepsilon\gg0$ such that
\begin{equation}\label{energyTeps}
	E_\varepsilon[\bm{u},\bm{u}_t](T_\varepsilon)\geq P_0[\bm{v}]-C_1\exp(-A/\varepsilon),
\end{equation}
for some $C_1>0$. 
Substituting \eqref{energy0} and \eqref{energyTeps} in \eqref{dissipative}, we deduce 
\begin{equation}\label{ut-small}
	\int_0^{T_\varepsilon}\int_a^b |\bm{u}_t|^2 dxdt\leq C_2\varepsilon\exp(-A/\varepsilon).
\end{equation}
Then, if $\varepsilon$ is small and $T_\varepsilon$ is very large, there is a very little excess of energy 
and the evolution of the solution is very slow.
We will prove that the solution maintains the same transitions of its initial datum $\bm{u}_0^\varepsilon$ for an exponentially large time.
Let us stress that \eqref{energyTeps} accounts for the improvements with respect to \cite{Folino}, 
because it implies \eqref{ut-small}, and so the $L^2$--norm of $\bm u_t$ is exponentially small.
The estimate \eqref{ut-small} permits to prove the main result of this paper.
\begin{rem}
Let us remark that $G$ is a positive-definite matrix for all $\bm{u}\in\mathbb R^m$, and the function $F$ vanishes only on a finite number of points.
As we already mentioned, the assumption \eqref{g-pos-def} is crucial in our proofs, 
because it implies the dissipative character of the system \eqref{hyp-al-ca-sys} 
and we can obtain the estimate \eqref{ut-small} on the time derivative of the solution.
In the case $G\equiv0$ we have a nonlinear wave equation of the form
\begin{equation*}
	\tau \bm{u}_{tt}=\varepsilon^2\bm{u}_{xx}+\bm{f}(\bm{u}),
\end{equation*}
which exhibits different dynamics (see \cite{Bel-Nov-Orl,Jerrard} and references therein, 
where the authors studied the case when $\bm f=-\nabla F$ and the potential $F$ vanishes on the unit circle).
We also underline that, in this paper, we consider the case of a bounded interval of the real line, 
and we use the boundedness of the domain in an essential way in some key estimates.
\end{rem}
The rest of the paper is organized as follows.
Section \ref{metastability}, the main section of the paper, is devoted to the analysis of metastability, and it contains the main result, Theorem \ref{main-hyp-sys}.
In Section \ref{transitionlayer} we construct an example of family of functions that has a transition layer structure.
These functions are metastable states for \eqref{hyp-al-ca-sys}-\eqref{Neumann-sys}.
Section \ref{layerdin} contains the study of the motion of the transition layers;
in particular, we prove that they move with exponentially small velocity (see Theorem \ref{thm:interface}).
Finally, in Appendix \ref{well-posed} we study the well-posedness of the initial boundary value problem 
\eqref{hyp-al-ca-sys}-\eqref{Neumann-sys}-\eqref{cond-iniz-sys} in the energy space ${H^1([a,b])}^m\times {L^2(a,b)}^m$.

\section{Metastability}\label{metastability}
In this section we study metastability of solutions to the nonlinear damped hyperbolic Allen--Cahn system \eqref{hyp-al-ca-sys}, 
where $\bm{u}\in\mathbb R^m$, with homogeneous Neumann boundary conditions \eqref{Neumann-sys}.
Fix $\bm{v}:[a,b]\rightarrow\{\bm{z}_1,\dots,\bm{z}_K\}$ having exactly $N$ jumps located at $a<\gamma_1<\gamma_2<\cdots<\gamma_N<b$. 
Fix $r$ so small that $B(\gamma_i,r)\subset[a,b]$ for any $i$ and 
\begin{equation*}
	B(\gamma_i,r)\cap B(\gamma_j,r)=\emptyset, \qquad \quad \mbox{ for } i\neq j.
\end{equation*}
Here and below $B(\gamma,r)$ is the open ball of center $\gamma$ and of radius $r$ in the relevant space. 
For $j=1,\dots,K$, denote by $\lambda_j$ (respectively, $\Lambda_j$) the minimum (resp. maximum) of the eigenvalues of $\nabla^2 F(\bm{z}_j)$.
If $\lambda=\displaystyle\min_j\lambda_j$ and $\Lambda=\displaystyle\max_j\Lambda_j$, we have for any $j=1,\dots,K$
\begin{equation}\label{hessian}
	0<\lambda|\bm y|^2\leq\nabla^2 F(\bm{z}_j)\bm y\cdot\bm y\leq\Lambda|\bm y|^2, \qquad \quad \forall\, \bm y\in\mathbb R^m.
\end{equation}
Now, we present the lower bound on the energy, which allows us to proof our main result. 
This result is purely variational in character and concerns only the functional $P_\varepsilon$; system \eqref{hyp-al-ca-sys} plays no role. 
The idea of the proof is the same of \cite[Lemma 2.1]{Grant}, we repeat it here for the convenience of the reader. 
\begin{prop}\label{prop:lower}
Assume that $F:\mathbb R^m\rightarrow\mathbb R$ satisfies \eqref{F1}-\eqref{F2}.
Let $\bm{v}:[a,b]\rightarrow\{\bm{z}_1,\dots,\bm{z}_K\}$ be a function having exactly $N$ jumps located at $a<\gamma_1<\gamma_2<\cdots<\gamma_N<b$
and let $A$ be a positive constant less than $r\sqrt{2\lambda}$. 
Then, there exist constants $C,\delta>0$ (depending only on $F,\bm{v}$ and $A$) such that, 
for $\varepsilon$ sufficiently small, if $\|{\bm{u}-\bm{v}}\|_{L^1}\leq\delta$, then
\begin{equation}\label{lower}
	P_\varepsilon[\bm{u}]\geq P_0[\bm{v}]-C\exp(-A/\varepsilon).
\end{equation}
\end{prop}
\begin{proof}
Let $Q$ be a compact set of $\mathbb R^m$ containing $F^{-1}(\{0\})$ in its interior and $\nu:=\sup\left\{\|\nabla^3F(\zeta)\|: \zeta\in Q\right\}$. 
Choose $\hat r>0$ and $\rho_1$ so small that $A\leq(r-\hat r)\sqrt{2\lambda-m\nu\rho_1}$ and 
that $B(\bm{z}_j,\rho_1)$ is contained in $Q$ for each $\bm{z}_j\in F^{-1}(\{0\})$. 
Choose $\rho_2$ so small that
\begin{align*}
	\inf&\bigl\{\phi(\bm{\xi}_1,\bm{\xi}_2): \bm{\xi}_1\notin B(\bm{z}_j,\rho_1), \bm{\xi}_2\in B(\bm{z}_j,\rho_2), \bm{z}_j\in F^{-1}(\{0\})\bigr\} \\
	&>\sup\bigl\{\phi(\bm{z}_j,\bm{\xi}_2): \bm{z}_j\in F^{-1}(\{0\}), \bm{\xi}_2\in B(\bm{z}_j,\rho_2)\bigr\},
\end{align*}
and $|\bm{z}_j-\bm{z}_l|>2\rho_2$ if $\bm{z}_j$ and $\bm{z}_l$ are different zeros of $F$. 

Now, let us focus our attention on $B(\gamma_i,r)$, a neighborhood of one of the transition points of $\bm{v}$. 
For convenience, let $\bm{v}_i^+:=\bm{v}(\gamma_i+r)$ and $\bm{v}_i^-:=\bm{v}(\gamma_i-r)$. 
We claim that there is some $r_+\in(0,\hat r)$ such that
\begin{equation*}
	|\bm{u}(\gamma_i+r_+)-\bm{v}_i^+|<\rho_2.
\end{equation*}
Indeed, if $|\bm{u}-\bm{v}|\geq\rho_2$ throughout $(\gamma_i,\gamma_i+\hat r)$, then  
\begin{equation*}
	\|\bm{u}-\bm{v}\|_{L^1}\geq\int_{\gamma_i}^{\gamma_i+\hat r} |\bm{u}-\bm{v}|\geq\hat r\rho_2>\delta,
\end{equation*}
if $\delta<\hat r\rho_2$, contrary to assumption on $\bm{u}$. 
Similarly, there is some $r_-\in(0,\hat r)$ such that
\begin{equation*}
	|\bm{u}(\gamma_i-r_-)-\bm{v}_i^-|<\rho_2.
\end{equation*}
Next, following \cite{Grant}, consider the unique minimizer $\bm{z}:[\gamma_i+r_+,\gamma_i+r]\rightarrow\mathbb R^m$ 
of the functional $P_\varepsilon[\bm{z};\gamma_i+r_+,\gamma_i+r]$ subject to the boundary condition 
\begin{equation*}
	\bm{z}(\gamma_i+r_+)=\bm{u}(\gamma_i+r_+).
\end{equation*}
If the range of $\bm{z}$ is not contained in $B(\bm{v}_i^+,\rho_1)$, then
\begin{align}
	P_\varepsilon[\bm{z};\gamma_i+r_+,\gamma_i+r] & \geq\inf\bigl\{\phi(\bm{z}(\gamma_i+r_+),\xi): \xi\notin B(\bm{v}_i^+,\rho_1)\bigr\} \nonumber\\
	&\geq\phi(\bm{z}(\gamma_i+r_+),\bm{v}_i^+), \label{E>fi}
\end{align}
by the choice of $r_+$ and $\rho_2$. 
Suppose, on the other hand, that the range of $\bm{z}$ is contained in $B(\bm{v}_i^+,\rho_1)$. 
Then, the Euler--Lagrange equation for $\bm{z}$ is
\begin{align*}
	\bm{z}''(x)=\varepsilon^{-2}\nabla F(\bm{z}(x)), \quad \qquad x\in(\gamma_i+r_+,\gamma_i+r),\\
	\bm{z}(\gamma_i+r_+)=\bm{u}(\gamma_i+r_+), \quad \qquad \bm{z}'(\gamma_i+r)=0.
\end{align*}
Denoting by $\psi(x):=|\bm{z}(x)-\bm{v}_i^+|^2$, we have $\psi'(x)=2(\bm{z}-\bm{v}_i^+)\cdot\bm{z}'$ and 
\begin{equation*}
	\psi''(x)=2(\bm{z}-\bm{v}_i^+)\cdot\bm{z}''+2|\bm{z}'|^2\geq\frac2{\varepsilon^2}(\bm{z}-\bm{v}_i^+)\cdot\nabla F(\bm{z}(x)).
\end{equation*}
Since $|\bm{z}(x)-\bm{v}_i^+|\leq\rho_1$ for any $x\in[\gamma_i+r_+,\gamma_i+r]$, using Taylor's expansion 
\begin{equation*}
	\nabla F(\bm{z}(x))=\nabla F(\bm{v}_i^+)+\nabla^2 F(\bm{v}_i^+)(\bm{z}(x)-\bm{v}_i^+)+R=\nabla^2 F(\bm{v}_i^+)(\bm{z}(x)-\bm{v}_i^+)+R,
\end{equation*}
where $|R|\leq m\nu|\bm{z}-\bm{v}_i^+|^2/2$, we obtain
\begin{align*}
	\psi''(x) & \geq \frac{2}{\varepsilon^2}\nabla^2 F(\bm{v}_i^+)(\bm{z}(x)-\bm{v}_i^+)\cdot(\bm{z}(x)-\bm{v}_i^+)-\frac{m\nu}{\varepsilon^2}|\bm{z}(x)-\bm{v}_i^+|^3\\
	& \geq \frac{2\lambda}{\varepsilon^2}|\bm{z}(x)-\bm{v}_i^+|^2-\frac{m\nu\rho_1}{\varepsilon^2}|\bm{z}(x)-\bm{v}_i^+|^2\\
	& \geq \frac{\mu^2}{\varepsilon^2}\psi(x),
\end{align*}
where $\mu=A/(r-\hat r)$. 
Thus, $\psi$ satisfies
\begin{align*}
	\psi''(x)-\frac{\mu^2}{\varepsilon^2}\psi(x)\geq0, \quad \qquad x\in(\gamma_i+r_+,\gamma_i+r),\\
	\psi(\gamma_i+r_+)=|\bm{u}(\gamma_i+r_+)-\bm{v}_i^+|^2, \quad \qquad \psi'(\gamma_i+r)=0.
\end{align*}
We compare $\psi$ with the solution $\hat \psi$ of
\begin{align*}
	\hat\psi''(x)-\frac{\mu^2}{\varepsilon^2}\hat\psi(x)=0, \quad \qquad x\in(\gamma_i+r_+,\gamma_i+r),\\
	\hat\psi(\gamma_i+r_+)=|\bm{u}(\gamma_i+r_+)-\bm{v}_i^+|^2, \quad \qquad \hat\psi'(\gamma_i+r)=0,
\end{align*}
which can be explicitly calculated to be
\begin{equation*}
	\hat\psi(x)=\frac{|\bm{u}(\gamma_i+r_+)-\bm{v}_i^+|^2}{\cosh\left[\frac\mu\varepsilon(r-r_+)\right]}\cosh\left[\frac\mu\varepsilon(x-(\gamma_i+r))\right].
\end{equation*}
By the maximum principle, $\psi(x)\leq\hat\psi(x)$ so, in particular,
\begin{equation*}
	\psi(\gamma_i+r)\leq\frac{|\bm{u}(\gamma_i+r_+)-\bm{v}_i^+|^2}{\cosh\left[\frac\mu\varepsilon(r-r_+)\right]}\leq2\exp(-A/\varepsilon)|\bm{u}(\gamma_i+r_+)-\bm{v}_i^+|^2.
\end{equation*}
Then, we have 
\begin{equation}\label{|z-v+|<exp}
	|\bm{z}(\gamma_i+r)-\bm{v}_i^+|\leq\sqrt2\exp(-A/2\varepsilon)\rho_2.
\end{equation}
Now, by using Taylor's expansion for $F(\bm{z}(x))$ and \eqref{hessian}, we obtain
\begin{align*}
	F(\bm{z}(x)) & =F(\bm{v}_i^+)+\nabla F(\bm{v}_i^+)\cdot(\bm{z}(x)-\bm{v}_i^+) \\
	& +\tfrac12\left(\nabla^2F(\bm{v}_i^+)(\bm{z}(x)-\bm{v}_i^+)\right)\cdot(\bm{z}(x)-\bm{v}_i^+)+o(|\bm{z}(x)-\bm{v}_i^+|^2)\\
	& \leq|\bm{z}(x)-\bm{v}_i^+|^2\left(\frac\Lambda2+\frac{o(|\bm{z}(x)-\bm{v}_i^+|^2)}{|\bm{z}(x)-\bm{v}_i^+|^2}\right).
\end{align*}
Similarly, one has 
\begin{equation*}
	F(\bm{z}(x))\geq|\bm{z}(x)-\bm{v}_i^+|^2\left(\frac\lambda2+\frac{o(|\bm{z}(x)-\bm{v}_i^+|^2)}{|\bm{z}(x)-\bm{v}_i^+|^2}\right). 
\end{equation*}
Therefore, since the range of $\bm{z}$ is contained in $B(\bm{v}_i^+,\rho_1)$, if $\rho_1$ is sufficiently small, then
\begin{equation}\label{F-quadratic}
	\tfrac14\lambda|\bm{z}(x)-\bm{v}_i^+|^2\leq F(\bm{z}(x))\leq\Lambda|\bm{z}(x)-\bm{v}_i^+|^2.
\end{equation}
Let us introduce the line segment
\begin{equation*}
	\hat{\bm{z}}(y):=\bm{v}_i^++\frac{y-a}{b-a}\left(\bm{z}(\gamma_i+r)-\bm{v}_i^+\right), \qquad \quad a\leq y\leq b.
\end{equation*}
We have $\hat{\bm{z}}(a)=\bm{v}_i^+$, $\hat{\bm{z}}(b)=\bm{z}(\gamma_i+r)$,
\begin{equation*}
	\hat{\bm{z}}'(y)=\frac1{b-a}(\bm{z}(\gamma_i+r)-\bm{v}_i^+), \qquad \quad |\hat{\bm{z}}(y)-\bm{v}_i^+|\leq|\bm{z}(\gamma_i+r)-\bm{v}_i^+|,
\end{equation*}
for any $y\in[a,b]$. 
Using \eqref{|z-v+|<exp} and \eqref{F-quadratic}, we obtain
\begin{align}
	\phi(\bm{v}_i^+,\bm{z}(\gamma_i+r)) & \leq\sqrt2\int_a^b\sqrt{F(\hat{\bm{z}}(y))}|\hat{\bm{z}}'(y)|dy \nonumber\\
	& \leq\frac{\sqrt{2\Lambda}}{b-a}|\bm{z}(\gamma_i+r)-\bm{v}_i^+)|\int_a^b|\hat{\bm{z}}(y)-\bm{v}_i^+|dy\nonumber\\
	& \leq\sqrt{2\Lambda}|\bm{z}(\gamma_i+r)-\bm{v}_i^+)|^2\nonumber\\
	& \leq2\sqrt{2\Lambda}\,\rho_2^2\,\exp(-A/\varepsilon). \label{fi<exp}
\end{align}
From \eqref{fi<exp} it follows that, for some constant $C>0$, 
\begin{align}
	P_\varepsilon[\bm{z};\gamma_i+r_+,\gamma_i+r] & \geq\phi(\bm{z}(\gamma_i+r_+),\bm{z}(\gamma_i+r)) \nonumber\\
	& \geq \phi(\bm{z}(\gamma_i+r_+),\bm{v}_i^+)-\phi(\bm{v}_i^+,\bm{z}(\gamma_i+r)) \nonumber\\
	& \geq \phi(\bm{z}(\gamma_i+r_+),\bm{v}_i^+)-\tfrac{C}{2N}\exp(-A/\varepsilon). \label{E>fi-exp}
\end{align}
Combining \eqref{E>fi} and \eqref{E>fi-exp}, we get that the constrained minimizer $\bm{z}$ of the proposed variational problem satisfies
\begin{equation*}	
	P_\varepsilon[\bm{z};\gamma_i+r_+,\gamma_i+r]\geq\phi(\bm{z}(\gamma_i+r_+),\bm{v}_i^+)-\tfrac{C}{2N}\exp(-A/\varepsilon).
\end{equation*}
The restriction of $\bm{u}$ to $[\gamma+r_+,\gamma+r]$ is an admissible function, so it must satisfy the same estimate
\begin{equation*}
	\begin{aligned}
	P_\varepsilon[\bm{u};\gamma_i+r_+,\gamma_i+r]&\geq P_\varepsilon[\bm{z};\gamma_i+r_+,\gamma_i+r]\\
	&\geq\phi(\bm{u}(\gamma_i+r_+),\bm{v}_i^+)-\tfrac{C}{2N}\exp(-A/\varepsilon).
	\end{aligned}
\end{equation*}
Considering the interval $[\gamma_i-r,\gamma_i-r_-]$, we obtain a similar estimate. Hence,
\begin{align*}
	P_\varepsilon[\bm{u};\gamma_i-r,\gamma_i+r] & =P_\varepsilon[\bm{u};\gamma_i-r,\gamma_i-r_-]+P_\varepsilon[\bm{u};\gamma_i-r_-,\gamma_i+r_+] \\
	& \quad +P_\varepsilon[\bm{u};\gamma_i+r_+,\gamma_i+r] \\
	& \geq\phi(\bm{v}_i^-,\bm{u}(\gamma_i-r_-))-\tfrac{C}{2N}\exp(-A/\varepsilon) \\
	& \quad +\phi(\bm{u}(\gamma_i-r_-),\bm{u}(\gamma_i+r_+))\\ 
	& \quad + \phi(\bm{u}(\gamma_i+r_+),\bm{v}_i^+)-\tfrac{C}{2N}\exp(-A/\varepsilon) \\
	& \geq \phi(\bm{v}(\gamma_i-r),\bm{v}(\gamma_i+r))-\tfrac{C}N\exp(-A/\varepsilon).
\end{align*}
These estimates hold for any $i=1,\dots,N$. Assembling all of these estimates, we have
\begin{equation*}
	P_\varepsilon[\bm{u}]\geq\sum_{i=1}^NP_\varepsilon[\bm{u};\gamma_i-r,\gamma_i+r]\geq P_0[\bm{v}]-C\exp(-A/\varepsilon),
\end{equation*}
and the proof is complete.
\end{proof}
Let us stress that Proposition \ref{prop:lower} extends and improves Proposition 2.1 of \cite{Folino}.
The sharp estimate \eqref{lower} is crucial in the proof of our main result.
Thanks to the equality \eqref{energy/eps-var} for the modified energy and the lower bound \eqref{lower}, 
we can use the energy approach in the study of the nonlinear damped hyperbolic Allen--Cahn system \eqref{hyp-al-ca-sys} 
with homogeneous Neumann boundary conditions \eqref{Neumann-sys} and initial data \eqref{cond-iniz-sys}. 
Let us proceed as in the scalar case $m=1$.

Regarding the initial data \eqref{cond-iniz-sys}, we assume that $\bm{u}_0$, $\bm{u}_1$ depend on $\varepsilon$ and 
\begin{equation}\label{|u_0-v|_1}
	\lim_{\varepsilon\rightarrow 0} \|\bm{u}_0^\varepsilon-\bm{v}\|_{{}_{L^1}}=0.
\end{equation}
In addition, we suppose that there exist constants $A\in(0,r\sqrt{2\lambda})$ and $\hat\varepsilon>0$ such that,
for all $\varepsilon\in(0,\hat\varepsilon)$, at the time $t=0$, the modified energy \eqref{energy/eps} satisfies
\begin{equation}\label{energy<E[v]+exp}
	E_\varepsilon[\bm{u}_0^\varepsilon, \bm{u}_1^\varepsilon]\leq P_0[\bm{v}]+C\exp(-A/\varepsilon),
\end{equation}
for some constant $C>0$.
The condition \eqref{|u_0-v|_1} fixes the number of transitions and their relative positions as $\varepsilon\to0$. 
The condition \eqref{energy<E[v]+exp} requires that the energy at the time $t=0$ exceeds at most $C\exp(-A/\varepsilon)$ the minimum possible to have these $N$ transitions.
Using \eqref{dissipative} and Proposition \ref{prop:lower}, we can prove the following result.

\begin{prop}\label{prop:L2-norm}
Assume that $G$ satisfies \eqref{g-pos-def} and that $\bm{f}=-\nabla F$ with $F$ satisfying \eqref{F1}-\eqref{F2}. 
Let $\bm{u}^\varepsilon$ be solution of \eqref{hyp-al-ca-sys}-\eqref{Neumann-sys}-\eqref{cond-iniz-sys} 
with initial data $\bm{u}_0^{\varepsilon}$, $\bm{u}_1^{\varepsilon}$ satisfying \eqref{|u_0-v|_1} and \eqref{energy<E[v]+exp}.
Then, there exist positive constants $\varepsilon_0, C_1, C_2>0$ (independent on $\varepsilon$) such that
\begin{equation}\label{L2-norm}
	\int_0^{C_1\varepsilon^{-1}\exp(A/\varepsilon)}\|\bm{u}_t^\varepsilon\|^2_{{}_{L^2}}dt\leq C_2\varepsilon \exp(-A/\varepsilon),
\end{equation}
for all $\varepsilon\in(0,\varepsilon_0)$.
\end{prop}

\begin{proof}
Let $\varepsilon_0>0$ so small that for all $\varepsilon\in(0,\varepsilon_0)$, \eqref{energy<E[v]+exp} holds and 
\begin{equation}\label{1/2delta}
	\|\bm{u}_0^\varepsilon-\bm{v}\|_{{}_{L^1}}\leq\frac12\delta,
\end{equation}
where $\delta$ is the constant of Proposition \ref{prop:lower}. 
Let $T_\varepsilon>0$.
We claim that if
\begin{equation}\label{claim1}
	\int_0^{T_\varepsilon}\|\bm{u}_t^\varepsilon\|_{{}_{L^1}}dt\leq\frac12\delta,
\end{equation}
then there exists $C_2>0$ such that
\begin{equation}\label{claim2}
	E_\varepsilon[\bm{u}^\varepsilon, \bm{u}_t^\varepsilon](T_\varepsilon)\geq P_0[\bm{v}]-C_2\exp(-A/\varepsilon).
\end{equation}
Indeed, $E_\varepsilon[\bm{u}^\varepsilon, \bm{u}_t^\varepsilon](T_\varepsilon)\geq P_\varepsilon[\bm{u}^\varepsilon](T_\varepsilon)$ and 
inequality \eqref{claim2} follows from Proposition \ref{prop:lower} if $\|\bm{u}^\varepsilon(\cdot,T_\varepsilon)-\bm{v}\|_{{}_{L^1}}\leq\delta$.  
By using triangle inequality, \eqref{1/2delta} and \eqref{claim1}, we obtain
\begin{equation*}
	\|\bm{u}^\varepsilon(\cdot,T_\varepsilon)-\bm{v}\|_{{}_{L^1}}\leq\|\bm{u}^\varepsilon(\cdot,T_\varepsilon)-\bm{u}_0^\varepsilon\|_{{}_{L^1}}+\|\bm{u}_0^\varepsilon-\bm{v}\|_{{}_{L^1}}
	\leq\int_0^{T_\varepsilon}\|\bm{u}_t^\varepsilon\|_{{}_{L^1}}+\frac12\delta\leq\delta.
\end{equation*}
Substituting \eqref{claim2} and \eqref{energy<E[v]+exp} in \eqref{dissipative}, one has 
\begin{equation}\label{L2-norm-Teps}
	\int_0^{T_\varepsilon}\|\bm{u}_t^\varepsilon\|^2_{{}_{L^2}}dt\leq C_2\varepsilon \exp(-A/\varepsilon),
\end{equation}
It remains to prove that inequality \eqref{claim1} holds for $T_\varepsilon\geq C_1\varepsilon^{-1}\exp(A/\varepsilon)$.
If 
\begin{equation*}
	\int_0^{+\infty}\|\bm{u}_t^\varepsilon\|_{{}_{L^1}}dt\leq\frac12\delta,
\end{equation*}
there is nothing to prove. Otherwise, choose $T_\varepsilon$ such that
\begin{equation*}
	\int_0^{T_\varepsilon}\|\bm{u}_t^\varepsilon\|_{{}_{L^1}}dt=\frac12\delta.
\end{equation*}
Using H\"older's inequality and \eqref{L2-norm-Teps}, we infer
\begin{equation*}
	\frac12\delta\leq[T_\varepsilon(b-a)]^{1/2}\biggl(\int_0^{T_\varepsilon}\|\bm{u}_t^\varepsilon\|^2_{{}_{L^2}}dt\biggr)^{1/2}\leq\bigl[T_\varepsilon(b-a)C_2\varepsilon \exp(-A/\varepsilon)\bigr]^{1/2}.
\end{equation*}
It follows that there exists $C_1>0$ such that
\begin{equation*}
	T_\varepsilon\geq C_1\varepsilon^{-1}\exp(A/\varepsilon),
\end{equation*}
and the proof is complete.
\end{proof}

Now, we can prove our main result.

\begin{thm}\label{main-hyp-sys}
Assume that $G$ satisfies \eqref{g-pos-def} and that $\bm{f}=-\nabla F$ with $F$ satisfying \eqref{F1}-\eqref{F2}. 
Let $\bm{u}^\varepsilon$ be solution of \eqref{hyp-al-ca-sys}-\eqref{Neumann-sys}-\eqref{cond-iniz-sys} 
with initial data $\bm{u}_0^{\varepsilon}$, $\bm{u}_1^{\varepsilon}$ satisfying \eqref{|u_0-v|_1} and \eqref{energy<E[v]+exp}.
Then, for any $s>0$
\begin{equation}\label{limit-sys}
	\sup_{0\leq t\leq s\exp(A/\varepsilon)}\|\bm{u}^\varepsilon(\cdot,t)-\bm{v}\|_{{}_{L^1}}\xrightarrow[\varepsilon\rightarrow0]{}0.
\end{equation}
\end{thm}

\begin{proof}
Fix $s>0$. Triangle inequality gives
\begin{equation}\label{triangle}
	\|\bm{u}^\varepsilon(\cdot,t)-\bm{v}\|_{{}_{L^1}}\leq\|\bm{u}^\varepsilon(\cdot,t)-\bm{u}_0^\varepsilon\|_{{}_{L^1}}+\|\bm{u}_0^\varepsilon-\bm{v}\|_{{}_{L^1}},
\end{equation}
for all $t\in[0,s\exp(A/\varepsilon)]$. 
The last term of inequality \eqref{triangle} tends to $0$ by assumption \eqref{|u_0-v|_1}, for the first one we have
\begin{equation*}
	\sup_{0\leq t\leq s\exp(A/\varepsilon)}\|\bm{u}^\varepsilon(\cdot,t)-\bm{u}_0^\varepsilon\|_{{}_{L^1}}\leq\int_0^{s\exp(A/\varepsilon)}\|\bm{u}_t^\varepsilon\|_{{}_{L^1}}dt.
\end{equation*}
Taking $\varepsilon$ so small that $s\leq C_1\varepsilon^{-1}$, we can apply Proposition \ref{prop:L2-norm} and deduce that
\begin{align}
	\int_0^{s\exp(A/\varepsilon)}\|\bm{u}_t^\varepsilon\|_{{}_{L^1}}dt&\leq[s\exp(A/\varepsilon)(b-a)]^{1/2}\biggl(\int_0^{s\exp(A/\varepsilon)}\|\bm{u}_t^\varepsilon\|^2_{{}_{L^2}}dt\biggr)^{1/2} \notag\\
	&\leq[s\exp(A/\varepsilon)(b-a)]^{1/2}\bigl[C_2\varepsilon \exp(-A/\varepsilon)\bigr]^{1/2}\notag\\
	&\leq\sqrt{C_2(b-a)s\varepsilon}. \label{int-u_t<eps^1/2}
\end{align}
Combining \eqref{|u_0-v|_1}, \eqref{triangle}, \eqref{int-u_t<eps^1/2} and by passing to the limit as $\varepsilon\to0$, we obtain \eqref{limit-sys}.
\end{proof}

\section{Example of transition layer structure}\label{transitionlayer}
In this section we construct an example of functions satisfying the assumptions \eqref{|u_0-v|_1} and \eqref{energy<E[v]+exp}.
Recall that, fixed $\bm{v}:[a,b]\rightarrow\{\bm{z}_1,\dots,\bm{z}_K\}$ having exactly $N$ jumps located at $a<\gamma_1<\gamma_2<\cdots<\gamma_N<b$, 
a family of functions $\bm{u}^\varepsilon$ has a transition layer structure if
\begin{equation}\label{trans-layer}
	\lim_{\varepsilon\rightarrow 0} \|\bm{u}_0^\varepsilon-\bm{v}\|_{{}_{L^1}}=0 \qquad \textrm{and} \qquad 
	P_\varepsilon[\bm{u}^\varepsilon]\leq P_0[\bm{v}]+C\exp(-A/\varepsilon).
\end{equation}
Then, in other words, the assumption \eqref{|u_0-v|_1} and \eqref{energy<E[v]+exp} are equivalent to
$\bm{u}_0^\varepsilon$ has a transition layer structure and the $L^2-$norm of $\bm{u}_1^\varepsilon$ is exponentially small.
Indeed, applying Proposition \ref{prop:lower} on $\bm{u}_0^\varepsilon$, one obtains for $\varepsilon$ sufficiently small
\begin{equation}\label{eq:ut}
	\tau\int_a^b|\bm{u}^\varepsilon_1(x)|^2dx\leq C{\varepsilon}\exp(-A/\varepsilon).
\end{equation}
Theorem \ref{main-hyp-sys}, roughly speaking, says that if $\bm{u}_0^\varepsilon$ has a transition layer structure and $\bm{u}_1^\varepsilon$ satisfies \eqref{eq:ut},
then $\bm{u}^\varepsilon(\cdot,t)$ maintains the transition layer structure for an exponentially large time. 
Moreover, the time derivative $\bm u_t$ satisfies \eqref{eq:ut} for an exponentially large time.

Let us construct a family of functions having a transition layer structure.
In the scalar case $m=1$, we can use the unique solution to the boundary value problem 
\begin{equation*}
	\varepsilon^2\Phi''+f(\Phi)=0, \qquad
	\Phi(0)=0, \qquad
	\Phi(x)\rightarrow\pm1 \quad \mbox{ as } \quad x\rightarrow\pm\infty,
\end{equation*}
and define the family $u^\varepsilon_0$ as
\begin{equation*}
	u_0^\varepsilon(x):=\Phi\bigl((x-\gamma_i)(-1)^{i+1}\bigr) \qquad \textrm{for } \, x\in[\gamma_{i-1/2},\gamma_{i+1/2}], \qquad i=1,\dots,N,
\end{equation*}
where
\begin{equation*}
	\gamma_{i+1/2}:=\frac{\gamma_i+\gamma_{i+1}}2, \qquad i=1,\dots, N-1, \qquad \gamma_{1/2}=a, \quad \gamma_{N+1/2}=b.
\end{equation*}
Note that $u_0^\varepsilon$ is a $H^1$ function with a piecewise continuous first derivative that jumps at $\gamma_{i+1/2}$ for $i=1\dots,N-1$, 
that $u^\varepsilon_0$ has a transition layer structure
and that $\Phi(x)=w(x/\varepsilon)$, where $w$ solves the Cauchy problem
\begin{equation*}
	\left\{\begin{aligned}
	& w'=\sqrt{2F(w)} \\
	& w(0)=0. 
	\end{aligned}\right. 
\end{equation*}
In the simplest example $F(w)=\frac14(w^2-1)^2$, we have $w(x)=\tanh(x/\sqrt2)$.

For $m>1$, we focus the attention on a fixed transition point $\gamma_i$ and we use again the notation 
$\bm{v}_i^+:=\bm{v}(\gamma_i+r)$ and $\bm{v}_i^-:=\bm{v}(\gamma_i-r)$. 
To construct a family $\bm{u}_0^\varepsilon$ having a transition layer structure, we use the following result of Grant \cite{Grant}.

\begin{lem}\label{lem:path}
Let $F:\mathbb R^m\rightarrow\mathbb R$ be a function satisfying \eqref{F1}-\eqref{F2}. 
Then, for any two zeros $\bm{z}_i$, $\bm{z}_j$ of $F$, there is a Lipschitz continuous path $\bm{\psi}_{ij}$ from $\bm{z}_i$ to $\bm{z}_j$, 
parametrized by a multiple of Euclidean arclength, such that $\phi(\bm{z}_i,\bm{z}_j)= J[\bm{\psi}_{ij}]$.
Moreover, there exists a constant $c>0$ such that 
\begin{equation*}
	\begin{aligned}
	& |\bm{\psi}_{ij}(w)-\bm{z}_i|\geq c(w-a) \qquad\qquad \textrm{for} \quad  w\approx a, \\
	& |\bm{\psi}_{ij}(w)-\bm{z}_j|\geq c(b-w) \qquad\qquad \textrm{for} \quad  w\approx b.
	\end{aligned}
\end{equation*}
\end{lem}

For the proof of this result see \cite[Lemma 3.2]{Grant}.

Denote by $\bm\psi_i:[a,b]\rightarrow\mathbb R^m$ the optimal path from $\bm{v}_i^-$ to $\bm{v}_i^+$ as described in Lemma \ref{lem:path}
and let $\sigma_i$ be the Euclidean arclength of $\bm\psi_i$, that is $|\bm\psi'_i(x)|=\sigma_i$ for all $x\in[a,b]$.
Assume, without loss of generality, that the path do not pass through any zero of $F$ (except at the endpoints of the path) and consider the solution of the Cauchy problem
\begin{equation}\label{ODE}
	\left\{\begin{aligned}
	& w'=\sigma_i^{-1}\sqrt{2F(\bm\psi_i(w))} \\
	& w(0)=\frac{b-a}2. 
	\end{aligned}\right. 
\end{equation}
There exists a unique $C^1$ solution  $w:\mathbb R\rightarrow(a,b)$ of \eqref{ODE}, because $\sqrt F$ and $\bm\psi_i$ are Lipschitz continuous, 
and $F$ satisfies \eqref{F-quadratic}. 
Indeed, 
\begin{equation*}
	\begin{aligned}
	& \sqrt{F(\bm\psi_i(w))}\leq\sigma_i\sqrt{\Lambda}|w-a| \qquad\qquad \textrm{for} \quad  w\approx a, \\
	& \sqrt{F(\bm\psi_i(w))}\leq\sigma_i\sqrt{\Lambda}|w-b| \qquad\qquad \textrm{for} \quad  w\approx b.
	\end{aligned}
\end{equation*}
Then, we deduce that
\begin{equation*}
	\lim_{x\to-\infty}w(x)=a \qquad  \textrm{and} \qquad  \lim_{x\to+\infty}w(x)=b. 
\end{equation*}
Now, we define $\bm{u}_0^\varepsilon:=\bm{v}$ outside of $\bigcup_{i=1}^N B(\gamma_i,r)$ and in $B(\gamma_i,r)$ we use the solution of \eqref{ODE}.
In order to construct a continuous function, let us define 
\begin{equation}\label{u0eps1}
	\bm{u}_0^\varepsilon(x):=\bm\psi_i\bigl(w((x-\gamma_i)/\varepsilon)\bigr) \qquad  \textrm{for } \, x\in[\gamma_i-r+\varepsilon,\gamma_i+r-\varepsilon],
\end{equation}
and use a line segment to connect $\bm\psi_i\bigl(w(1-r/\varepsilon)\bigr)$ with $\bm{v}_i^-$ and $\bm\psi_i\bigl(w(r/\varepsilon-1)\bigr)$ with $\bm{v}_i^+$.
Hence, we have
\begin{equation}\label{u0eps2}
	\bm{u}_0^\varepsilon(x):=\left\{
	\begin{aligned}
	\bm{v}_i^-+\frac{x-\gamma_i+r}{\varepsilon}\left(\bm\psi_i\bigl(w(1-r/\varepsilon)\bigr)-\bm{v}_i^-\right), \quad  x\in(\gamma_i-r,\gamma_i-r+\varepsilon),\\
	\bm{v}_i^++\frac{\gamma_i+r-x}{\varepsilon}\left(\bm\psi_i\bigl(w(r/\varepsilon-1)\bigr)-\bm{v}_i^+\right), \quad  x\in(\gamma_i+r-\varepsilon,\gamma_i+r). 
	\end{aligned}\right.
\end{equation}
By joining \eqref{u0eps1} and \eqref{u0eps2}, we conclude the definition of $\bm{u}_0^\varepsilon$ in $B(\gamma_i,r)$.
Note that $\bm{u}_0^\varepsilon$ is a piecewise continuously differentiable function and, for \eqref{u0eps1} one has 
\begin{equation*}
	|(\bm{u}_0^\varepsilon)'(x)|=\frac{\sigma_i}\varepsilon|w'((x-\gamma_i)/\varepsilon)| \qquad  \textrm{for } \,[\gamma_i-r+\varepsilon,\gamma_i+r-\varepsilon].
\end{equation*}
Using this equality and \eqref{ODE}, we deduce
\begin{equation}\label{eq:F(u0)}
	\frac12\varepsilon^2|(\bm{u}_0^\varepsilon)'|^2=F(\bm{u}_0^\varepsilon) \qquad\quad \textrm{in } [\gamma_i-r+\varepsilon,\gamma_i+r-\varepsilon].
\end{equation}
Now, let us show that the family of functions $\bm{u}_0^\varepsilon$ has a transition layer structure, i.e. $\bm{u}_0^\varepsilon$ satisfies \eqref{trans-layer}.  
The $L^1$ requirement follows from the dominated convergence theorem. Let us prove the energy requirement. 

\begin{prop}
Assume that $F:\mathbb R^m\rightarrow\mathbb R$ satisfies \eqref{F1}-\eqref{F2}.
Let $\bm{v}:[a,b]\rightarrow\{\bm{z}_1,\dots,\bm{z}_K\}$ be a function having exactly $N$ jumps located at $a<\gamma_1<\gamma_2<\cdots<\gamma_N<b$
and let $\bm{u}_0^\varepsilon$ be a function such that $\bm{u}_0^\varepsilon:=\bm{v}$ outside of $\bigcup_{i=1}^N B(\gamma_i,r)$ 
and $\bm{u}_0^\varepsilon$ satisfies \eqref{u0eps1}, \eqref{u0eps2} in $B(\gamma_i,r)$. 
For all $A\in\big(0,c\sigma^{-1}r\sqrt{2\lambda}\big)$ (where $c$ is the constant introduced in Lemma \ref{lem:path} and $\displaystyle\sigma:=\max_i \sigma_i$), 
there exist constants $\varepsilon_0, C>0$ such that, if $\varepsilon\in(0,\varepsilon_0)$, then  
\begin{equation}\label{u0-energy}
	P_\varepsilon[\bm{u}_0^\varepsilon]\leq P_0[\bm{v}]+C\exp(-A/\varepsilon).
\end{equation}
\end{prop}

\begin{proof}
By definition, we have 
\begin{equation*}
	P_\varepsilon[\bm{u}_0^\varepsilon]=\sum_{i=1}^NP_\varepsilon[\bm{u}_0^\varepsilon;\gamma_i-r,\gamma_i+r].
\end{equation*}
Then, we must estimate the energy functional in $B(\gamma_i,r)$. 
For definitions \eqref{u0eps1} and \eqref{u0eps2}, we split
\begin{equation*}
	P_\varepsilon[\bm{u}_0^\varepsilon;\gamma_i-r,\gamma_i+r]:=I_1+I_2+I_3,
\end{equation*}
where 
\begin{align*}
	I_1&:=\int_{\gamma_i-r}^{\gamma_i-r+\varepsilon}\left[\frac\varepsilon2 |(\bm{u}_0^\varepsilon)'(x)|^2+\frac{F(\bm{u}_0^\varepsilon(x))}\varepsilon\right]dx,\\
	I_2&:=\int_{\gamma_i-r+\varepsilon}^{\gamma_i+r-\varepsilon}\left[\frac\varepsilon2 |(\bm{u}_0^\varepsilon)'(x)|^2+\frac{F(\bm{u}_0^\varepsilon(x))}\varepsilon\right]dx,\\
	I_3&:=\int_{\gamma_i+r-\varepsilon}^{\gamma_i+r}\left[\frac\varepsilon2 |(\bm{u}_0^\varepsilon)'(x)|^2+\frac{F(\bm{u}_0^\varepsilon(x))}\varepsilon\right]dx.		
\end{align*}
To start with, we estimate the term $I_2$. By using \eqref{eq:F(u0)} and changing variable $y=w((x-\gamma_i)/\varepsilon)$, we obtain
\begin{equation*}
	I_2=\int_{\gamma_i-r+\varepsilon}^{\gamma_i+r-\varepsilon}\frac{2F(\bm{u}_0^\varepsilon(x))}\varepsilon dx
	=\sqrt2\int_{w(1-r/\varepsilon)}^{w(r/\varepsilon-1)}\sqrt{F(\bm{\psi}_i(y))}|\bm{\psi}'_i(y)|dy.  
\end{equation*}
By definition $\bm{\psi}_i$ is an optimal path from $\bm{v}_i^-$ to $\bm{v}_i^+$ and as a consequence 
\begin{equation}\label{eq:I2}
	I_2\leq\sqrt2\int_a^b\sqrt{F(\bm{\psi}_i(y))}|\bm{\psi}'_i(y)|dy=\phi(\bm{v}_i^-,\bm{v}_i^+).
\end{equation}
Next, we estimate $I_1$. 
We have
\begin{equation*}
	I_1:=\int_{-r}^{-r+\varepsilon}\biggl[\frac1{2\varepsilon}|\bm\psi_i\bigl(w(1-r/\varepsilon)\bigr)-\bm{v}_i^-|^2
	+\frac1\varepsilon F\Bigl(\bm{v}_i^-+\frac{x+r}{\varepsilon}\bigl(\bm\psi_i\bigl(w(1-r/\varepsilon)\bigr)-\bm{v}_i^-\bigr)\Bigr)\biggr]dx.
\end{equation*}
To estimate the latter term, for $\varepsilon$ sufficiently small, we use \eqref{F-quadratic} to obtain
\begin{equation*}
	F\Bigl(\bm{v}_i^-+\frac{x+r}{\varepsilon}\bigl(\bm\psi_i\bigl(w(1-r/\varepsilon)\bigr)-\bm{v}_i^-\bigr)\Bigr)\leq\Lambda|\bm\psi_i\bigl(w(1-r/\varepsilon)\bigr)-\bm{v}_i^-|^2.
\end{equation*}
Thanks to this bound and the Lipschitz continuity of $\bm{\psi}_i$, one has 
\begin{equation}\label{diseq:I1}
	I_1\leq C|w(1-r/\varepsilon)-a|^2.
\end{equation}
Here and in what follows, $C$ is a positive constant (independent on $\varepsilon$) whose value may change from line to line.
In order to estimate the right hand side of \eqref{diseq:I1}, let us use Lemma \ref{lem:path} and \eqref{F-quadratic}.
Since $w(x)\to a$ as $x\to-\infty$ and $\bm{\psi}_i(a)=\bm{v}_i^-$, there exists $x_1>0$ sufficiently large so that
\begin{equation*}
	w'(x)\geq(\sigma_i\sqrt2)^{-1}\sqrt{\lambda}|\bm{\psi}_i(w(x))-\bm{v}_i^-|\geq c(\sigma\sqrt2)^{-1}\sqrt{\lambda}(w(x)-a),
\end{equation*}
for all $x\leq-x_1$, where $c>0$ is the constant introduced in Lemma \ref{lem:path}. 
Using the notation $c_1:=c(\sigma\sqrt2)^{-1}\sqrt{\lambda}$ and multiplying by $\exp(-c_1x)$, one has
\begin{equation*}
	\bigl(\exp(-c_1x)w(x)\bigr)'\geq-ac_1(\exp(-c_1x),
\end{equation*}
for all $x\leq-x_1$.
By integrating the latter inequality, we infer
\begin{equation}\label{w-a}
	w(x)-a\leq C\exp(c_1x),
\end{equation}
for all $x\leq-x_1$. 
If $\varepsilon$ is so small that $1-r/\varepsilon\leq-x_1$, by substituting \eqref{w-a} into \eqref{diseq:I1}, we obtain
\begin{equation}\label{eq:I1}
	I_1\leq C\exp(2c_1(1-r/\varepsilon))\leq C\exp(-2c_1r/\varepsilon)\leq C\exp(-A/\varepsilon),
\end{equation}
for all positive constant $A\leq2c_1r\leq c\sigma^{-1}r\sqrt{2\lambda}$.
In a similar way, we can obtain the estimate for $I_3$.
For all $A\in\big(0,c\sigma^{-1}r\sqrt{2\lambda}\big)$, we have
\begin{equation}\label{eq:I3}
	I_3\leq C|w(r/\varepsilon-1)-b|^2\leq C\exp(-A/\varepsilon).
\end{equation}
Combining \eqref{eq:I2}, \eqref{eq:I1} and \eqref{eq:I3}, we deduce
\begin{equation*}
	P_\varepsilon[\bm{u}_0^\varepsilon;\gamma_i-r,\gamma_i+r]\leq\phi(\bm{v}_i^-,\bm{v}_i^+)+C\exp(-A/\varepsilon),
\end{equation*}
and as a trivial consequence we have \eqref{u0-energy}.
\end{proof}
Then, we can conclude that $\bm{u}_0^\varepsilon$ has a transition layer structure and if
the $L^2$--norm of $\bm u^\varepsilon_1$ is exponentially small (see \eqref{eq:ut}) the solution of \eqref{hyp-al-ca-sys}-\eqref{Neumann-sys}-\eqref{cond-iniz-sys}
evolves very slow and maintains the same transition layer structure of the initial datum $\bm{u}_0^\varepsilon$ for an exponentially long time.

\section{Layer dynamics}\label{layerdin}
In this section we study the motion of the transition layers and 
we show that Theorem~\ref{main-hyp-sys} implies that the movement of the layers is extremely slow.
To do this, we adapt the strategy already used in \cite{Grant,Folino}.
Before stating the main result of the section, we need some definitions. 
If $\bm{v}:[a,b]\to\mathbb R^m$ is a step function with jumps at $\gamma_1,\gamma_2,\ldots,\gamma_N$, then its {\it interface} $I[\bm{v}]$ is defined by 
\begin{equation*}
	I[\bm{v}]:=\{\gamma_1,\gamma_2,\ldots,\gamma_N\}.
\end{equation*}
For an arbitrary function $\bm{u}:[a,b]\rightarrow\mathbb{R}^m$ and an arbitrary closed subset $D\subset\mathbb R^m\backslash F^{-1}(\{0\})$,
the {\it interface} $I_D[\bm{u}]$ is defined by
\begin{equation*}
	I_D[\bm{u}]:=\bm{u}^{-1}(D).
\end{equation*}
Finally, for any $A,B\subset\mathbb{R}$ the {\it Hausdorff distance} $d(A,B)$ between $A$ and $B$ is defined by 
\begin{equation*}
	d(A,B):=\max\biggl\{\sup_{\alpha\in A}d(\alpha,B),\,\sup_{\beta\in B}d(\beta,A)\biggr\},
\end{equation*}
where $d(\beta,A):=\inf\{|\beta-\alpha|: \alpha\in A\}$. 

Now we can state the main result of this section.
\begin{thm}\label{thm:interface}
Assume that $G$ satisfies \eqref{g-pos-def} and that $\bm{f}=-\nabla F$ with $F$ satisfying \eqref{F1}-\eqref{F2}. 
Let $\bm{u}^\varepsilon$ be solution of \eqref{hyp-al-ca-sys}-\eqref{Neumann-sys}-\eqref{cond-iniz-sys} 
with initial data $\bm{u}_0^{\varepsilon}$, $\bm{u}_1^{\varepsilon}$ satisfying \eqref{|u_0-v|_1} and \eqref{energy<E[v]+exp}.
Given $\delta_1\in(0,r)$ and a closed subset $D\subset\mathbb R^m\backslash F^{-1}(\{0\})$, set
\begin{equation*}
	T_\varepsilon(\delta_1)=\inf\{t:\; d(I_D[\bm u^\varepsilon(\cdot,t)],I_D[\bm{u}_0^\varepsilon])>\delta_1\}.
\end{equation*}
There exists $\varepsilon_0>0$ such that if $\varepsilon\in(0,\varepsilon_0)$ then
\begin{equation}\label{T-interface}
	T_\varepsilon(\delta_1)> \exp(A/\varepsilon).
\end{equation}
\end{thm}

In order to prove Theorem \ref{thm:interface}, we use the following result, 
that is, as Proposition \ref{prop:lower}, purely variational in character and concerns only the functional $P_\varepsilon$.

\begin{lem}\label{lem:interface}
Assume that $F:\mathbb R^m\rightarrow\mathbb R$ satisfies \eqref{F1}-\eqref{F2}. 
Let $\bm{v}:[a,b]\rightarrow\{\bm{z}_1,\dots,\bm{z}_K\}$ be a function having exactly $N$ jumps located at $a<\gamma_1<\gamma_2<\cdots<\gamma_N<b$.
Given $\delta_1\in(0,r)$ and a closed subset $D\subset\mathbb R^m\backslash F^{-1}(\{0\})$, 
there exist $\varepsilon_0,\rho>0$ such that for all functions $\bm{u}^\varepsilon:[a,b]\to\mathbb R^m$ satisfying
\begin{equation}\label{eq:u-v}
	\|\bm{u}^\varepsilon-\bm{v}\|_{{}_{L^1}}<\tfrac12\rho\,\delta_1
\end{equation}
and 
\begin{equation}\label{eq:E[u]}
	P_\varepsilon[\bm{u}^\varepsilon]\leq P_0[\bm{v}]+ 2N\sup\{\phi(\bm{z}_j,\bm{\xi}) : \bm{z}_j\in F^{-1}(\{0\}), \bm{\xi}\in B(\bm{z}_j,\rho)\},
\end{equation}
for all $\varepsilon\in(0,\varepsilon_0)$, we have
\begin{equation}\label{lem:d-interfaces}
	d(I_D[\bm{u}^\varepsilon], I[\bm{v}])<\tfrac12\delta_1.
\end{equation}
\end{lem}
\begin{proof}
Choose $\rho>0$ small enough that
\begin{align*}
	\inf\{ & \phi(\bm{\xi}_1,\bm{\xi}_2) : \bm{z}_j\in F^{-1}(\{0\}), \bm{\xi}_1\in K, \bm{\xi}_2\in B(\bm{z}_j,\rho)\}\\
	&>4N\sup\{\phi(\bm{z}_j,\bm{\xi}_2) : \bm{z}_j\in F^{-1}(\{0\}), \bm{\xi}_2\in B(\bm{z}_j,\rho)\}.
\end{align*}
By reasoning as in Proposition \ref{prop:lower}, we obtain that for each $i$ there exist
\begin{equation*}
	x^-_{i}\in(\gamma_i-\delta_1/2,\gamma_i) \qquad \textrm{and} \qquad x^+_{i}\in(\gamma_i,\gamma_i+\delta_1/2)
\end{equation*}
such that
\begin{equation*}
	|\bm{u}^\varepsilon(x^-_{i})-\bm{v}(x^-_{i})|<\rho \qquad \textrm{and} \qquad |\bm{u}^\varepsilon(x^+_{i})-\bm{v}(x^+_{i})|<\rho.
\end{equation*}
Suppose that \eqref{lem:d-interfaces} is violated. 
Then, we deduce
\begin{align}
	P_\varepsilon[\bm{u}^\varepsilon]&\geq\sum_{i=1}^N P_\varepsilon[\bm{u}^\varepsilon;x^-_{i},x^+_{i}]\notag\\
	&+\inf\{\phi(\bm{\xi}_1,\bm{\xi}_2) : \bm{z}_j\in F^{-1}(\{0\}), \bm{\xi}_1\in K, \bm{\xi}_2\in B(\bm{z}_j,\rho)\}. \label{diseq:E1}
\end{align}
On the other hand, triangle inequality gives 
\begin{equation*}
	\phi\bigl(\bm{v}(x^+_{i}),\bm{v}(x^-_{i})\bigr)\leq\phi\bigl(\bm{v}(x^+_{i}),\bm{u}^\varepsilon(x^+_{i})\bigr)
	+\phi\bigl(\bm{u}^\varepsilon(x^+_{i}),\bm{u}^\varepsilon(x^-_{i})\bigr)+\phi\bigl(\bm{u}^\varepsilon(x^-_{i}),\bm{v}(x^-_{i})\bigr)
\end{equation*}
and as a consequence
\begin{align*}
	\phi\bigl(\bm{u}^\varepsilon(x^-_{i}),\bm{u}^\varepsilon(x^+_{i})\bigr)\geq &\, \phi\bigl(\bm{v}(x^+_{i}),\bm{v}(x^-_{i})\bigr)\\
	& -2\sup\{\phi(\bm{z}_j,\bm{\xi}_2) : \bm{z}_j\in F^{-1}(\{0\}), \bm{\xi}_2\in B(\bm{z}_j,\rho)\}. 
\end{align*}
Substituting the latter bound in \eqref{diseq:E1} and recalling that 
\begin{equation*}
	P_\varepsilon[\bm{u}^\varepsilon;x^-_{i},x^+_{i}]\geq\phi\bigl(\bm{u}^\varepsilon(x^-_{i}),\bm{u}^\varepsilon(x^+_{i})\bigr),
\end{equation*}
we infer
\begin{align*}
	P_\varepsilon[\bm{u}^\varepsilon]\geq P_0[\bm{v}]&-2N\sup\{\phi(\bm{z}_j,\bm{\xi}_2) : \bm{z}_j\in F^{-1}(\{0\}), \bm{\xi}_2\in B(\bm{z}_j,\rho)\}\\
	&+\inf\{\phi(\bm{\xi}_1,\bm{\xi}_2) : \bm{z}_j\in F^{-1}(\{0\}), \bm{\xi}_1\in K, \bm{\xi}_2\in B(\bm{z}_j,\rho)\}.
\end{align*}
For the choice of $\rho$ and assumption \eqref{eq:E[u]}, we obtain
\begin{align*}
	P_\varepsilon[\bm{u}^\varepsilon]>P_0[\bm{v}]+2N\sup\{\phi(\bm{z}_j,\bm{\xi}_2) : \bm{z}_j\in F^{-1}(\{0\}), \bm{\xi}_2\in B(\bm{z}_j,\rho)\}
	\geq P_\varepsilon[\bm{u}^\varepsilon],
\end{align*}
which is a contradiction. Hence, the bound \eqref{lem:d-interfaces} is true.
\end{proof}
The previous result and Theorem \ref{main-hyp-sys} permits to prove Theorem \ref{thm:interface}.
\begin{proof}[Proof of Theorem \ref{thm:interface}]
Let $\varepsilon_0>0$ so small that the assumptions on the initial data \eqref{|u_0-v|_1}, \eqref{energy<E[v]+exp} 
imply that $\bm{u}_0^\varepsilon$ satisfy \eqref{eq:u-v} and \eqref{eq:E[u]} for all $\varepsilon\in(0,\varepsilon_0)$.
From Lemma \ref{lem:interface} it follows that
\begin{equation}\label{interfaces-u0}
	d(I_D[\bm{u}_0^\varepsilon], I[\bm{v}])<\tfrac12\delta_1.
\end{equation}
Now, we apply the same reasoning to $\bm{u}^\varepsilon(\cdot,t)$ for all $t\leq\exp(A/\varepsilon)$.
Assumption \eqref{eq:u-v} is satisfied for Theorem \ref{main-hyp-sys},
while \eqref{eq:E[u]} holds because $E_\varepsilon[\bm{u}^\varepsilon,\bm{u}^\varepsilon_t](t)$ is a nonincreasing function of $t$.
Then,
\begin{equation}\label{interfaces-u}
	d(I_D[\bm{u}^\varepsilon(t)], I[\bm{v}])<\tfrac12\delta_1
\end{equation}
for all $t\in(0,\exp(A/\varepsilon))$. 
Combining \eqref{interfaces-u0} and \eqref{interfaces-u}, we obtain
\begin{equation*}
	d(I_D[\bm{u}^\varepsilon(t)],I_D[\bm{u}_0^\varepsilon])<\delta_1
\end{equation*}
for all $t\in(0,\exp(A/\varepsilon))$ and the proof is complete.
\end{proof}
Then, the velocity of the transition layers is exponentially small.
Thanks to Theorem \ref{main-hyp-sys} and Theorem \ref{thm:interface}, we obtain exponentially slow motion. 
In \cite{FLM}, similar results have been obtained in the scalar case, 
by using a different method, the dynamical approach of Carr and Pego \cite{Carr-Pego}.

\vspace{0.5cm}
\textbf{Acknowledgments}. I am very grateful to C. Lattanzio and C. Mascia for their helpful advices.

\appendix

\section{Existence and uniqueness}\label{well-posed}
In this appendix we study the well-posedness of the following initial boundary problem
\begin{equation}\label{IBVP-sys}
	\begin{cases}
	\tau \bm{u}_{tt}+G(\bm{u})\bm{u}_t=\varepsilon^2\bm{u}_{xx}+\bm{f}(\bm{u}) \qquad \qquad & x\in[a,b], \, t>0,\\
	\bm{u}(x,0)=\bm{u}_0(x) & x\in[a,b],\\
	\bm{u}_t(x,0)=\bm{u}_1(x) & x\in[a,b],\\
	\bm{u}_x(a,t)=\bm{u}_x(b,t)=0 & t>0,
	\end{cases}
\end{equation} 
where $\bm{u}(x,t)\in\mathbb R^m$, $G:\mathbb R^m\rightarrow\mathbb  R^{m\times m}$, $\bm{f}:\mathbb R^m\rightarrow\mathbb R^m$ and $\varepsilon,\tau>0$.
The strategy we will use is standard and based on the semigroup theory for solutions of differential equations on Hilbert spaces 
(see Cazenave and Haraux \cite{Cazenave}, and Pazy \cite{Pazy}).
Following the ideas of the scalar case $m=1$ (cfr. \cite{Folino}) and setting $\bm{y}=(\bm{u},\bm{v})=(\bm{u},\bm{u}_t)$, 
we rewrite the first equation of \eqref{IBVP-sys} as a first order evolution equation 
\begin{equation}\label{evol-eq}
	\bm{y}_t=A_m\bm{y}+\bm{\Phi}_m(\bm{y}),
\end{equation}
where 
\begin{align}
	A_m\bm{y}&:=\left(\begin{array}{cc} 0_m & \mathbb{I}_m \\ \varepsilon^2\tau^{-1}\partial^2_x\mathbb{I}_m & 0_m \end{array} \right)\bm{y}-\bm{y} \label{eq:Am}\\
	\bm{\Phi}_m(\bm{y})&:=\bm{y}+\dfrac1\tau\left(\begin{array}{c} 0 \\ \bm{f}(\bm{u})-G(\bm{u})\bm{v} \end{array}\right).\label{eq:Fim}
\end{align}
The unknown $\bm{y}$ is considered as a function of a real (positive) variable $t$ with values on the function space $X^m={H^1([a,b])}^m\times{L^2(a,b)}^m$ with scalar product
\begin{equation*}
	\langle(\bm{u},\bm{v}),(\bm{w},\bm{z})\rangle_X:= \int_a^b (\varepsilon^2\bm{u}_x\cdot\bm{w}_x+\tau\bm{u}\cdot\bm{w}+\tau\bm{v}\cdot\bm{z})dx, 
\end{equation*}
that is equivalent to the usual scalar product in ${H^1([a,b])}^m\times{L^2(a,b)}^m$. 

\begin{prop}\label{A-m-dis}
The linear operator $A_m:D(A_m)\subset X^m\rightarrow X^m$ defined by \eqref{eq:Am} with 
\begin{equation}\label{dom-A}
	D(A_m)=\bigl\{(\bm{u},\bm{v})\in{H^2([a,b])}^m\times{H^1([a,b])}^m : \bm{u}_x(a)=\bm{u}_x(b)=0\bigr\},
\end{equation}
is m-dissipative with dense domain.
\end{prop}
The proof is just a vector notation of the scalar case $m=1$ (see \cite[Proposition A.3]{Folino}).

Given a matrix $B\in\mathbb R^{m\times m}$, we denote by $\|\cdot\|_\infty$ the matrix norm induced by the vector norm $|\bm{u}|_\infty=\max|u_j|$ on $\mathbb R^m$
\begin{equation*}
	\|B\|_\infty:=\max_{1\leq i\leq m}\sum_{j=1}^m|b_{ij}|.
\end{equation*}
We suppose that $\bm{f}\in C(\mathbb R^m,\mathbb R^m)$ and
\begin{equation}\label{f-lip}
	|\bm{f}(\bm{x}_1)-\bm{f}(\bm{x}_2)|\leq L_1(K)|\bm{x}_1-\bm{x}_2|, \qquad \quad \forall\, \bm{x}_1,\bm{x}_2\in B_K.
\end{equation}
Here and below $B_K$ is the open ball of center $0$ and of radius $K$ in the relevant space. 
Regarding $G$, we suppose that $G\in C(\mathbb R^m,\mathbb R^{m\times m})$ and
\begin{equation}\label{g-lip}
	\|G(\bm{x}_1)-G(\bm{x}_2)\|_\infty\leq L_2(K)|\bm{x}_1-\bm{x}_2|, \qquad \quad \forall\, \bm{x}_1,\bm{x}_2\in B_K.
\end{equation}
Then, $\bm{f}$ and $G$ are locally Lipschitz continuous functions. 
If $\bm{f}$ satisfies \eqref{f-lip} and if $\bm{\mathcal{F}}$ is the operator defined by $(\bm{\mathcal{F}}(\bm{u}))(x):=\bm{f}(\bm{u}(x))$, 
then $\bm{\mathcal{F}}$ maps ${H^1([a,b])}^m$ into ${L^2(a,b)}^m$ and there exists $C(K)>0$ such that
\begin{equation}\label{F-lip}
	\|\bm{\mathcal{F}}(\bm{u}_1)-\bm{\mathcal{F}}(\bm{u}_2)\|_{{}_{L^2}}\leq C_1(K)\|\bm{u}_1-\bm{u}_2\|_{{}_{L^2}}, \qquad \quad \forall\, \bm{u}_1,\bm{u}_2\in B_K,
\end{equation}
where $\displaystyle{\|\bm{u}\|}^2_{{}_{L^2}}:=\int_a^b |\bm{u}|^2dx$. 

Moreover, we have
\begin{equation}\label{g(u)v}
	\|G(\bm{u})\bm{v}\|^2_{{}_{L^2}}\leq \int_a^b \|G(\bm{u})\|_{{}_{\infty}}^2|\bm{v}|^2dx\leq\max_x{\|G(\bm{u}(x))\|}_{{}_\infty}^2{\|\bm{v}\|}^2_{{}_{L^2}},
\end{equation}
for all $(\bm{u},\bm{v})\in X^m$. Using \eqref{g-lip} and \eqref{g(u)v}, we obtain
\begin{equation}\label{g(u)v-2}
	\|(G(\bm{u}_1)-G(\bm{u}_2))\bm{v}\|^2_{{}_{L^2}}\leq C_2(K){\|\bm{u}_1-\bm{u}_2\|}_{{}_{L^\infty}}^2{\|\bm{v}\|}^2_{{}_{L^2}},
\end{equation}
for all $\bm{u}_1,\bm{u}_2\in B_K$. 
It follows that the function $\bm{\Phi}_m$ defined by \eqref{eq:Fim} is a Lipschitz continuous function on bounded subsets of $X^m$. 
Indeed, for all $\bm{y}_1=(\bm{u}_1,\bm{v}_1)$, $\bm{y}_2=(\bm{u}_2,\bm{v}_2)\in X^m$ we have
\begin{align*}
	\|\bm{\Phi}_m(\bm{y}_1)-\bm{\Phi}_m(\bm{y}_2)\|_{{}_{X^m}}\leq &\; \|\bm{y}_1-\bm{y}_2\|_{{}_{X^m}} \\
	& +C\left(\|\bm{\mathcal{F}}(\bm{u}_1)-\bm{\mathcal{F}}(\bm{u}_2)\|_{{}_{L^2}}+\|G(\bm{u}_1)\bm{v}_1-G(\bm{u}_2)\bm{v}_2\|_{{}_{L^2}}\right).
\end{align*}
Let $K:=\max\{\|\bm{y}_1\|_{X^m},\|\bm{y}_2\|_{{}_{X^m}}\}$. 
We have that
\begin{align*}
	\|G(\bm{u}_1)\bm{v}_1-G(\bm{u}_2)\bm{v}_2\|_{{}_{L^2}}\leq\,& \|G(\bm{u}_1)(\bm{v}_1-\bm{v}_2)\|_{{}_{L^2}}+\|(G(\bm{u}_1)-G(\bm{u}_2))\bm{v}_2\|_{{}_{L^2}}\\
	\leq\,& C(K)(\|\bm{v}_1-\bm{v}_2\|_{{}_{L^2}}+\|\bm{u}_1-\bm{u}_2\|_{{}_{H^1}}),
\end{align*}
where the last inequality holds because $G$ is locally Lipschitz continuous and $H^1([a,b])\subset L^\infty([a,b])$ with continuous inclusion. 
From this inequality and \eqref{F-lip}, it follows that there exists a constant $L(K)$ (depending on $K$) such that
\begin{equation*}
	\|\bm{\Phi}_m(\bm{y}_1)-\bm{\Phi}_m(\bm{y}_2)\|_{{}_{X^m}}\leq L(K)\|\bm{y}_1-\bm{y}_2\|_{{}_{X^m}}.
\end{equation*}
Therefore, we can proceed in the same way of the scalar case $m=1$. 
For all $\bm{x}\in X^m$ the Cauchy problem \eqref{evol-eq}, with $A_m$ and $\bm{\Phi}_m$ defined by \eqref{eq:Am}-\eqref{dom-A} and \eqref{eq:Fim}, 
$\bm{f},G$ locally Lipschitz continuous and initial data $\bm{y}(0)=\bm{x}$ has a unique mild solution on $[0,T(\bm{x}))$, that is a function $\bm{y}\in C([0,T(\bm{x}),X^m)$ solving the problem
\begin{equation*}	
	\bm{y}(t)=S_m(t)\bm{x}+\int_0^t S_m(t-s)\bm{\Phi}_m(\bm{y}(s))ds, \qquad \quad \forall\, t\in[0,T(\bm{x})),
\end{equation*}
where $(S_m(t))_{t\geq0}$ is the contraction semigroup in $X^m$, generated by $A_m$. 
In particular, if $\bm{x}=(\bm{u}_0,\bm{u}_1)\in D(A_m)$, then $\bm{y}=(\bm{u},\bm{u}_t)$ is a classical solution, 
that is a solution of \eqref{IBVP-sys} for $t\in[0,T(\bm{x}))$ satisfying
\begin{equation*}
	(\bm{u},\bm{u}_t)\in C([0,T(\bm{x})),D(A_m))\cap C^1([0,T(\bm{x})),X^m).
\end{equation*}
In order to show the global existence of the solution we define the energy
\begin{equation}\label{energy}
	E[\bm{u},\bm{u}_t](t):=\int_a^b\biggl[\frac\tau2|\bm{u}_t(x,t)|^2+\frac{\varepsilon^2}2 |\bm{u}_x(x,t)|^2+F(\bm{u}(x,t))\biggr]dx,
\end{equation}
where $\bm{f}(\bm{u})=-\nabla F(\bm{u})$. 
Observe that the energy \eqref{energy} is well-defined for mild solutions $(\bm{u},\bm{u}_t)\in C([0,T],X^m)$. 
Using the same procedure of \cite{Folino}, we can prove the following result.
\begin{prop}\label{prop-var-ener}
Assume that $\bm{f}$ and $G$ are locally Lipschitz continuous. If $(\bm{u},\bm{u}_t)\in C([0,T],X^m)$ is a mild solution, then 
\begin{equation}\label{variazione-energia}
	\int_0^T\int_a^b G(\bm{u})\bm{u}_t\cdot\bm{u}_t dxdt=E[\bm{u},\bm{u}_t](0)-E[\bm{u},\bm{u}_t](T).
\end{equation}
\end{prop}
\begin{proof}
Let $T>0$ and $\bm{u}$ be a classical solution of the problem \eqref{IBVP-sys}. 
Taking the scalar product with $\bm{u_t}$ and integrating on $[a,b]\times(0,T)$, we have 
\begin{equation*}
	\int_0^T\int_a^b\left(\tau\bm{u}_t\cdot\bm{u}_{tt}+G(\bm{u})\bm{u}_t\cdot\bm{u}_t\right)dxdt=
	\int_0^T\int_a^b\left(\varepsilon^2 \bm{u}_t\cdot\bm{u}_{xx}+\bm{f}(\bm{u})\cdot\bm{u}_t\right)dxdt.
\end{equation*}
Using integration by parts and the homogeneous Neumann boundary conditions, we obtain
\begin{align*}
	\int_a^b\tau\bm{u}_t\cdot\bm{u}_{tt}\,dx& =\tau\langle\bm{u}_t,\bm{u}_{tt}\rangle_{{}_{L^2}}=
	\frac{d}{dt}\left(\frac\tau2\|\bm{u}_t\|^2_{{}_{L^2}}\right),\\
	\int_a^b\varepsilon^2\bm{u}_t\cdot\bm{u}_{xx}\,dx& =-\varepsilon^2\int_a^b\bm{u}_{tx}\cdot\bm{u}_xdx=
	-\frac{d}{dt}\left(\frac{\varepsilon^2}2\|\bm{u}_x\|^2_{{}_{L^2}}\right).
\end{align*}
Therefore, it follows that
\begin{align*}
	\int_0^T\int_a^b G(\bm{u})\bm{u}_t\cdot\bm{u}_t\,dxdt &=\int_a^b\left[\frac\tau2 |\bm{u}_t(x,0)|^2-\frac\tau2|\bm{u}_t(x,T)|^2\right]dx\\
	&+\int_a^b\left[\frac{\varepsilon^2}2|\bm{u}_x(x,0)|^2-\frac{\varepsilon^2}2|\bm{u}_x(x,T)|^2\right]dx\\
	&+\int_a^b\left[F(\bm{u}(x,0))-F(\bm{u}(x,T))\right]dx. 
\end{align*}
Using the definition of energy \eqref{energy}, we have \eqref{variazione-energia} for classical solutions.

If $\bm{x}\in D(A_m)$ the solution is classical and \eqref{variazione-energia} holds. 
For $\bm{x}\in X^m\backslash D(A_m)$, we use the continuous dependence on the initial data of the solution (cfr. \cite[Proposition 4.3.7]{Cazenave}).
Let us consider $\bm{x}_n\in D(A_m)$ such that $\bm{x}_n\rightarrow \bm{x}$ in $X^m$. 
For the corresponding solution $\bm{y}_n=(\bm{u}_n,(\bm{u}_n)_t)$, \eqref{variazione-energia} is satisfied; 
by passing to the limit and using Proposition 4.3.7 of \cite{Cazenave}, we obtain \eqref{variazione-energia} for $\bm{y}=(\bm{u},\bm{u}_t)$.
\end{proof}
If we assume that $G(\bm{u})$ is positive semi-definite for all $\bm{u}\in\mathbb R^m$, 
then the energy is a nonincreasing function of $t$ along the solutions of \eqref{IBVP-sys}. 
Furthermore, if $G(\bm{u})$ is positive definite for all $\bm{u}\in\mathbb R^m$, then there exists a constant $\alpha>0$ such that
\begin{equation*}
	\alpha\int_0^T\int_a^b |\bm{u}_t|^2 dxdt\leq E[\bm{u},\bm{u}_t](0)-E[\bm{u},\bm{u}_t](T).
\end{equation*}
Therefore, the initial boundary value problem \eqref{IBVP-sys} is globally well-posed in the energy space ${H^1([a,b])}^m\times {L^2(a,b)}^m$.
\begin{thm}\label{global existence}
Assume that $\bm{f},G$ are locally Lipschitz continuous, 
\begin{equation}\label{hp-g}
	G(\bm{x})\bm{y}\cdot \bm{y}\geq0, \qquad \quad \forall\,\bm{x},\bm{y}\in\mathbb R^m
\end{equation}
and that there is $L>0$ such that for any $|\bm x|>L$
\begin{equation}\label{hp-f}
	F(\bm{x})\geq C|\bm x|^2, \qquad \quad \mbox{ for some }\, C\in\mathbb{R},
\end{equation}
where $\bm{f}(\bm x):=-\nabla F(\bm x)$. 
Then, for any $(\bm{u}_0,\bm{u}_1)\in {H^1([a,b])}^m\times {L^2(a,b)}^m$ there exists a unique mild solution of \eqref{IBVP-sys}
\begin{equation*}
	(\bm{u},\bm{u}_t)\in C\left([0,\infty), {H^1([a,b])}^m\times {L^2(a,b)}^m\right).
\end{equation*}
\end{thm}
Thanks to Proposition \ref{A-m-dis} and Proposition \ref{prop-var-ener}, 
the proof is just a vector notation of the scalar case (see \cite[Theorem A.7]{Folino}).

\end{document}